\documentclass[12pt]{amsart}

\usepackage{latexsym,amsfonts,amssymb,epsfig,verbatim}
\usepackage{amsmath,amsthm,amssymb,latexsym,graphics,textcomp}
\usepackage{eucal,eufrak}
\usepackage{graphicx,pinlabel}
\input{xy}
\xyoption{all}
\input{epsf.tex}

\newcommand{\nc}{\newcommand}
\nc{\nt}{\newtheorem}
\nc{\dmo}{\DeclareMathOperator}

\theoremstyle{plain}
\newtheorem{theorem}{Theorem}[section]
\newtheorem{lemma}[theorem]{Lemma}
\newtheorem{corollary}[theorem]{Corollary}

\newtheorem{proposition}[theorem]{Proposition}

\nc{\R}{\mathbb R}
\nc{\C}{\mathbb C}
\nc{\Z}{\mathbb Z}
\nc{\N}{\mathbb N}

\nc{\F}{\mathcal F}

\dmo{\ext}{ext}
\renewcommand{\int}{\mathrm{int}}
\dmo{\Area}{Area}
\nc{\G}{\mathcal{G}}
\def\PSL{\mathrm{PSL}}
\def\TT{\mathcal{T}}
\def\T{\mathrm{Teich}}
\dmo{\Mod}{Mod}
\nc{\M}{\mathcal{M}}
\nc{\inj}{\mathrm{inj}}

\nc{\q}{\bf q}

\nc{\p}[1]{\medskip\paragraph{{\em #1}}}
\nc{\margin}[1]{\marginpar{\scriptsize #1}}

\title[Short geodesics in moduli space]{On the number and location of short geodesics in moduli space}

\begin{document}

\author{Christopher J.~Leininger and Dan Margalit}

\address{Dan Margalit \\ School of Mathematics\\ Georgia Institute of Technology \\ 686 Cherry St. \\ Atlanta, GA 30332 \\  margalit@math.gatech.edu}

\address{
Christopher J. Leininger\\ Dept. of Mathematics, University of Illinois at Urbana-Champaign \\ 273 Altgeld Hall, 1409 W. Green St. \\ Urbana, IL 61802\\ clein@math.uiuc.edu}

\thanks{The authors gratefully acknowledge support from the National
Science Foundation and the Sloan Foundation}


\begin{abstract}
A closed Teichm\"uller geodesic in the moduli space $\M_g$ of Riemann surfaces of genus $g$ is called {\em $L$-short}, if it has length at most $L/g$.  We show that for any $L > 0$ there exist $\epsilon_2 > \epsilon_1 > 0$, independent of $g$, so that the $L$-short geodesics in $\M_g$ all lie in the intersection of the $\epsilon_1$-thick part and the $\epsilon_2$-thin part.  We also estimate the number of $L$-short geodesics in $\M_g$, bounding this from above and below by polynomials in $g$ whose degrees depend on $L$ and tend to infinity as $L$ does.
\end{abstract}

\maketitle

\vspace{-.15in}

\section{Introduction}

Given $g \geq 1$, let $\M_g$ denote the moduli space of Riemann surfaces of genus $g$ equipped with the Teichm\"uller metric.  For any $L > 0$, we define
\[ \G_g(L) = \{ \text{closed geodesics in } \M_g \text{ of length at most } L/g \}. \]
We refer to the elements of $\G_g(L)$ as \emph{$L$-short geodesics}, or \emph{short geodesics} for short.

Ivanov \cite{Iv} and Arnoux--Yoccoz \cite{AY} showed that the set $\G_g(L)$ is finite for every $g \geq 1$ and $L > 0$.
Penner \cite{Pe} proved that there exists an $L_0$ so that $\G_g(L)$ is nonempty for all $g \geq 1$ and $L > L_0$.  In fact, Hironaka \cite{Hir} showed that we can take $L_0 = \log((3+\sqrt{5})/2) \approx 0.962$ for sufficiently large $g$; see also \cite{AD,KT}.

Given an interval $I \subset (0,\infty)$, let $\M_{g,I}$ be the subset of $\M_g$ consisting of those hyperbolic surfaces (Euclidean surfaces in the case $g=1$) in which the length of the shortest essential closed curve lies in $I$.  For example, the sets $\M_{g,(0,\epsilon]}$ and $\M_{g,[\epsilon,\infty)}$ are often called the \emph{$\epsilon$-thin part} and \emph{$\epsilon$-thick part} of $\M_g$, respectively.

Our first theorem provides a coarse description of the location of the set of short geodesics in $\M_g$. 

\begin{theorem}
\label{T:location}
Given $L > 0$ there exists $\epsilon_2 > \epsilon_1 > 0$ so that each element of $\G_g(L)$ lies in $\M_{g,[\epsilon_1,\epsilon_2]}$ for all $g \geq 1$.
\end{theorem}

Our second and third theorems concern the number of short geodesics in $\M_g$, counted as a function of $g$: the number of $L$-short geodesics in $\M_g$ is bounded from above and below by polynomials in $g$ whose degrees depend on $L$ and tend to infinity as $L$ does.

\begin{theorem}
\label{T:count}
Given $L > 0$ there exists a polynomial $P_L(g)$ so that
\[ \left| \G_g(L) \right| \leq P_L(g) \]
for all $g \geq 1$.
\end{theorem}

\begin{theorem}
\label{T:count2}
Given $d > 0$, there exists a polynomial $Q_d(g)$ of degree $d$, with positive leading coefficient, and $L>0$ so that
\[ \left| \G_g(L) \right| \geq Q_d(g) \]
for all $g \geq 1$.
\end{theorem}

\noindent {\bf Short geodesics as small-dilatation pseudo-Anosov mapping classes.} Let $S_g$ denote a closed, connected, orientable surface of genus $g$.  We will now give an interpretation of $\G_g(L)$ that is intrinsic to $S_g$.  For more details on pseudo-Anosov homeomorphisms, the mapping class group, and Teichm\"uller space, see \cite{primer}.

A homeomorphism $\phi : S_g \to S_g$ is \emph{pseudo-Anosov} if there are measured singular foliations $(\F_+,\mu_+)$ and $(\F_-,\mu_-)$, called the stable and unstable measured foliations, and a real number $\lambda(\phi) > 1$, called the \emph{dilatation}, so that
\[ \phi(\F_+,\mu_+) = \lambda(\phi)(\F_+,\mu_+) \quad \text{and} \quad \phi(\F_-) = \lambda(\phi)^{-1}(F_-,\mu_-). \]
When $g=1$, the foliations $\F_+$ and $\F_-$ are nonsingular, and $\phi$ is usually called \emph{Anosov}.  For ease of exposition, we will consider Anosov homeomorphisms to also be pseudo-Ansoov.

The mapping class group $\Mod(S_g)$ is the group of homotopy classes of homeomorphisms of $S_g$.  An element of $\Mod(S_g)$ is pseudo-Anosov if it has a pseudo-Anosov representative.  

There is a natural action of $\Mod(S_g)$ on Teichm\"uller space $\T(S_g)$, the space of isotopy classes of hyperbolic metrics on $S_g$, and the quotient is nothing other than moduli space:
\[ \M_g = \T(S_g)/\Mod(S_g). \]
The Teichm\"uller distance between two points of $\T(S_g)$ is $\log(K)/2$, where $K$ is the quasiconformal distortion between the two corresponding metrics on $S_g$, minimized over all representatives of the respective isotopy classes.  The group $\Mod(S_g)$ acts on $\T(S_g)$ properly discontinuously by isometries, and so there is an induced metric on $\M_g$, as above.

Each pseudo-Anosov $[\phi]$ in $\Mod(S_g)$ acts on $\T(S_g)$ by translation along a geodesic axis.  The translation length of $[\phi]$ is precisely $\log(\lambda(\phi))$ and the quotient of the axis descends to a closed geodesic in $\M_g$ of length $\log(\lambda(\phi))$.  Furthermore, conjugate pseudo-Anosov mapping classes define the same closed geodesic, and, moreover, every closed geodesic in $\M_g$ arises in this way.

We define the set of small dilatation pseudo-Anosov mapping classes as
\[ \Psi_g(L) = \{ [\phi] \in \Mod(S_g) \mid \phi \mbox{ is pseudo-Anosov and } \log(\lambda(\phi)) \leq L/g \}. \]
By the previous paragraph, there is a bijection between $\G_g(L)$ and $\Psi_g(L)/\Mod(S_g)$, the set of $\Mod(S_g)$-conjugacy classes of elements of $\Psi_g(L)$:
\[ \G_g(L) \leftrightarrow \Psi_g(L)/\Mod(S_g).\]
As such, both of our main theorems can be rephrased as statements about the set of real numbers $\lambda$ that arise as dilatations of pseudo-Anosov homeomorphisms of $S_g$.

\bigskip

\noindent {\bf Small dilatations and 3-manifolds.} Theorems~\ref{T:location} and~\ref{T:count} will be deduced from a finiteness result proven by the authors with Benson Farb, which we now recall.  Consider the set of all small dilatation pseudo-Anosov mapping classes of all closed surfaces:
\[ \Psi(L) = \bigcup_{g \geq 1} \Psi_g(L).\]
For each element of $\Psi(L)$, we define a new pseudo-Anosov homeomorphism by removing the singularities of the stable and unstable foliations and taking the restriction.  Let $\TT(L)$ denote the set of mapping tori that arise from these modified pseudo-Anosov maps, considered up to homeomorphism.

We have the following theorem; see \cite[Theorem 1.1]{FLM} and \cite[Theorem 6.2]{Ag}.

\begin{theorem} \label{T:universal}
For all $L > 0$, the set $\TT(L)$ is finite.
\end{theorem}

Because of Theorem~\ref{T:universal}, it is enough to prove Theorems~\ref{T:location} and~\ref{T:count} for the pseudo-Anosov homeomorphisms corresponding to a single element of $\TT(L)$.  We can then take maxima and minima of all of the resulting bounds in order to obtain the theorems.

\bigskip

\noindent {\bf Acknowledgments.} We would like to thank Jayadev Athreya, Greg Blekherman, Martin Bridgeman, Jeff Brock, Dick Canary, Benson Farb, Richard Kent, Curt McMullen, and Maryam Mirzakhani for helpful conversations.

\section{The Thurston norm} \label{S:thurston norm}

Theorem~\ref{T:universal} allows us to realize $\G_g(L)$ as the union of a finite number of sets, namely, the short geodesics arising from the different fibers of the finite set of 3-manifolds $\TT(L)$.  In order to leverage this theorem effectively, we will need a way of organizing the elements of $\G_g(L)$ coming from a particular 3-manifold $M$ of $\TT(L)$.  The Thurston norm on $H^1(M;\R)$ is well-suited to this purpose.

Let $M \in \TT(L)$.  By definition, $M$ is equal to the mapping torus $M_\phi$, where $\phi : S \to S$ is a pseudo-Anosov homeomorphism of a punctured surface $S$:
\[ M_\phi = S \times [0,1]/((\phi(x),0) \sim (x,1)). \]
We refer to the surface $S \subset M$ as a fiber, since it is a fiber in a fibration of $M$ over the circle.
A deep theorem of Thurston states that the mapping torus of any pseudo-Anosov homeomorphism, hence $M$, admits a finite-volume hyperbolic structure \cite{Ot}.  In this section, we will only use the observation that $M$ is atoroidal.

Fibers of $M$ represent elements of $H^1(M;\R)$.  More precisely, each fiber in a fibration of $M$ over $S^1$ determines and is determined up to isotopy by a homology class which is Poincar\'e dual to an integral element of $H^1(M;\R)$.  Furthermore, primitive integral elements of $H^1(M;\R)$ correspond to connected fibers.  In this way, we identify the set of isotopy classes of fibers in $M$ with a subset of the integral elements of $H^1(M;\R)$.  

Let $M$ be any finite-volume hyperbolic 3-manifold. Thurston defined a norm
\[ \| \cdot \| : H^1(M;\R) \to \R, \]
now called the \emph{Thurston norm}, and proved that the set of all fibers of the mapping torus $M_\phi$ has a convenient description in terms of $|| \cdot ||$.  We summarize the properties of the Thurston norm in the following theorem \cite{ThNorm}.  

\begin{theorem} \label{T:thurston}
Suppose $M$ is a finite volume hyperbolic 3-manifold.
\begin{itemize}
\item The unit ball in $H^1(M;\R)$ with respect to the Thurston norm is a compact polyhedron $B$.
\item There is a set of open top-dimensional faces $F_1,\dots,F_n$ of $B$ so that the fibers of $M$ exactly correspond to the integral elements of the union of the open cones $\R_+ \cdot F_i$.
\item The restriction of $\| \cdot \|$ to any cone $\R_+ \cdot F_i$ is equal to the restriction of a homomorphism $\psi_i:H^1(M;\R) \to \R$ with the property that $\psi_i(H^1(M;\Z)) \subseteq 2\Z$.
\item If $S$ is a fiber, then $\|S\| = -\chi(S)$.
\end{itemize}
\end{theorem}

The open faces $F_1,\ldots,F_n$ in Theorem~\ref{T:thurston} are called the \emph{fibered faces} of $M$.  We will often abuse notation by writing $S \in \R_+ \cdot F_i$ to mean that the cohomology class dual to the fiber $S$ lies in the cone over the fibered face $F_i$.

The homomorphisms $\psi_i$ can be described as follows \cite[Theorem 3]{ThNorm}.  Given a fiber $S \in \R_+ \cdot F_i$, the union of all fibers of the fibration defines a codimension $1$ foliation of $M$.  The tangent spaces to the leaves form a $2$-plane bundle $\tau_i$ on $M$, whose homotopy class only depends on $F_i$.  The relative Euler class $e(\tau_i)$, relative to the inward-pointing vector in a neighborhood of the cusp, is dual to an element of $H_1(M;\R)$ which, by pairing with $H^1(M;\R)$, defines a homomorphism to $\R$.  This is precisely $-\psi_i$:
\[ \psi_i(\eta) =  - e(\tau_i) \cdot \eta \]
for all $\eta \in H^1(M;\R)$.

\section{Counting short geodesics I}

We now apply Theorems~\ref{T:universal} and~\ref{T:thurston} in order to prove Theorem~\ref{T:count}, which states that, given $L > 0$ there exists a polynomial $P_L(g)$ so that
\[ \left| \G_g(L) \right| \leq P_L(g) \]
for all $g \geq 1$.

\begin{proof}[Proof of Theorem~\ref{T:count}]

Recall from the introduction that
\[ |\Psi_g(L)/\Mod(S_g)| = |\G_g(L)|. \]
Thus, given $L > 0$, it suffices find a polynomial $P_L(g)$ so that
\[ |\Psi_g(L)/\Mod(S_g)| \leq P_L(g) \]
for all $g$.

According to Theorem~\ref{T:universal}, the set of 3-manifolds $\TT(L)$ is finite.  For any $M \in \TT(L)$, let $b_1(M) = \dim(H^1(M ;\R))$ be the first Betti number of $M$.  

Let $B(r)$ denote the closed ball of radius $r$ around 0 in $H^1(M ;\R)$ with respect to the Thurston norm.  There is a polynomial $p_M(r)$  of degree $b_1(M)$ so that
\[  | H^1(M ;\Z) \cap B(r) | \leq p_M(r). \]

Let $(\phi:S \to S) \in \Psi_g(L)$.  By the Poincar\'e--Hopf index theorem, the number of singular points of the stable foliation for $\phi$ is at most $4g-4$.  Thus, if $S'$ denotes the surface obtained from $S$ by deleting these singular points, we have
\[ |\chi(S')| \leq 6g-6.\]
Let $\phi' : S' \to S'$ denote the restriction of $\phi$ to $S'$.  The map $\phi$ is completely determined up to conjugacy by the conjugacy class of $\phi'$, so it suffices to count the number of conjugacy classes of maps $\phi'$ arising from elements of $\Psi_g(L)$.  By the last statement of Theorem~\ref{T:thurston}, we have
\[ \|S'\| \leq 6g-6. \]
In other words, each $\phi \in \Psi_g(L)$ is, after deleting singular points, the monodromy of some fiber in the ball of radius $6g-6$ with respect to the Thurston norm of some $M \in \TT(L)$.  Thus, setting
\[ P_L(g) = \sum_{M \in \TT(L)} p_M(6g-6),\]
it follows that
\[ |\Psi_g(L)/\Mod(S_g)| \leq P_L(g),\]
as desired.
\end{proof}

\section{Two theorems of Fried about fibered faces} \label{S:more thurston norm}

Let $\phi:S \to S$ be a pseudo-Anosov homeomorphism.  The {\em suspension flow} $\phi_t$ determined by $S$ and $\phi$ is a flow on $M_\phi$ defined using the coordinates
\[ M_\phi = S \times [0,1]/(x,1) \sim (\phi(x),0), \]
and extending the local flow $(x,s) \mapsto (x,s+t)$ on $S \times [0,1]$ to $M_\phi$.  If $b_1(M_\phi) \geq 2$, then $M_\phi$ fibers in infinitely many ways and we obtain infinitely many different suspension flows on $M_\phi$.

We note that $\phi_t$ is transverse to $S$ and the first return map to $S$ is precisely the monodromy $\phi$.  The (unmeasured) stable and unstable foliations $\F_\pm$ for $\phi$ can therefore be suspended.  The result is a pair of $\phi_t$-invariant singular foliations on $M$ which we denote $\F_\pm^M$.  

Fried \cite[Theorem 7 and Lemma]{Fr} proved that monodromy of any other fiber in $\R_+ \cdot F$, the cone containing $S$, has the following description (see also \cite{LO}).

\begin{theorem} \label{T:friedtransverse}
Let $\phi:S \to S$ be a pseudo-Anosov homeomorphism with stable and unstable foliations $\F_\pm$. Let $\phi_t$ denote the suspension flow on $M_\phi$ determined by $S$ and $\phi$, and let $\F_\pm^M$ denote the $\phi_t$-suspensions of $\F_\pm$.  Let $F$ be the fibered face of $M_\phi$ with $S \in \R_+ \cdot F$.  Then for any fiber $\Sigma \in \R_+ \cdot F$, we can modify $\Sigma$ by isotopy so that
\begin{enumerate}
\item the fiber $\Sigma$ is transverse to $\phi_t$ and the first return map $\Sigma \to \Sigma$ is precisely the pseudo-Anosov monodromy associated to $\Sigma$, and
\item the intersections $\F_+^M \cap \Sigma$ and $\F_-^M \cap \Sigma$  are the stable and unstable foliations for $\varphi$, respectively.
\end{enumerate}
\end{theorem}

In Theorem~\ref{T:friedtransverse}, the foliations $\F_+^M$ and $\F_-^M$ are only topological foliations of $M_\phi$, and not transversely measured foliations.  In particular, the intersections $\F_+^M \cap \Sigma$ and $\F_-^M \cap \Sigma$ are only the topological stable and unstable foliations for a fiber $\Sigma$, and not the transversely measured foliations.   We will return to this issue in Section~\ref{S:continuity}.

\bigskip

Let $M$ be a finite-volume hyperbolic 3-manifold, and let $F$ be a fibered face of $M$.  Fried proved that there is a continuous function
\[ \Lambda_F : \R_+ \cdot F  \to (1,\infty), \]
with the property that for each fiber $S \in \R_+ \cdot F$, $\Lambda_F(S) = \lambda(\phi)$, where $\phi : S \to S$ is the monodromy.  We summarize the properties of $\Lambda_F$ in the following; see \cite[Theorem F]{Fr0} and \cite[Section 5]{Mc}.

\begin{theorem} \label{T:frieddilatation}
Let $M$ be a finite-volume hyperbolic 3-manifold.  For each fibered face $F$ of $M$, there exists a continuous function
\[ \Lambda_F:\R_+ \cdot F \to \R \]
with the following properties:
\begin{itemize}
\item For every $\eta \in \R_+ \cdot F$ and every $t > 0$, we have
\[ \Lambda_F(t \eta)  = \Lambda_F(\eta)^{1/t}.\]
\item For any fiber $S \in \R_+ \cdot F$ with monodromy $\phi$, we have:
\[ \Lambda_F(S) = \lambda(\phi).\]
\item For any sequence $\{ \eta_i \} \subset \R_+ \cdot F$ with a nonzero limit outside the open cone $\R_+ \cdot F$, we have 
\[ \Lambda_F(\eta_i) \to \infty.\]
\end{itemize}
\end{theorem}

\bigskip

\p{McMullen's proof of Penner's theorem} As observed by McMullen \cite{Mc}, Theorem~\ref{T:frieddilatation} can be used to prove Penner's theorem that there is an $L_0$ so that $\G_g(L_0)$ (equivalently, $\Psi_g(L_0)$) is nonempty for all $g \geq 1$.  To see this, let $\phi : S_2 \to S_2$ be a pseudo-Anosov homeomorphism whose action on $H_1(S_2;\R)$ fixes a nontrivial element and consider the mapping torus $M_\phi$.  Because $\phi$ has a nonzero fixed vector, we have $b_1(M_\phi) \geq 2$.  Let $F$ be the fibered face of $M_\phi$ with $S_2 \in \R_+ \cdot F$.  

Recall from Theorem~\ref{T:thurston} that the restriction of the Thurston norm to the cone over $F$ is given by the restriction of a homomorphism $\psi : H^1(M_\phi;\R) \to \R$ with $\psi(H^1(M_\phi;\Z)) \subseteq 2\Z$.  Since every integral class has even norm, and $\psi(S_2) = ||S_2|| = 2$, we have $\psi(H^1(M_\phi;\Z)) = 2\Z$. Let $\Sigma \in H^1(M_\phi;\Z)$ be an element of the kernel of $\psi$ which, together with $S_2$, is part of a basis for $H^1(M_\phi;\Z)$.  For large $g$, the primitive cohomology class
\[ \Sigma_g = (g-1) \cdot S_2 + \Sigma \]
lies in the cone $\R_+ \cdot F$.  Also, applying the last statement of Theorem~\ref{T:thurston}, we have
\[ \|\Sigma_g\| = \psi(\Sigma_g) = \psi((g-1) \cdot S_2 + \Sigma) = (g-1)\psi(S_2) = (g-1)\|S_2\| = 2g-2.\]
As $M_\phi$ is closed, each fiber $\Sigma_g$ is a closed surface.  Moreover, since each $\Sigma_g$ represents a primitive cohomology class, it is a connected surface.  Since $\|\Sigma_g\|=2g-2$, it follows from Theorem~\ref{T:thurston} that $\Sigma_g$ has genus $g$.

According to Theorems~\ref{T:thurston} and~\ref{T:frieddilatation}, the function
\[ \eta \mapsto \|\eta\| \log(\Lambda_F(\eta)) \]
is continuous on $\R_+ \cdot F$ and is constant on rays from the origin.  Furthermore, for every $\Sigma_g$, we have
\[ \|\Sigma_g\| \log(\Lambda_F(S_g)) = (2g-2)\log(\lambda(\phi_g)), \]
where $\phi_g:\Sigma_g \to \Sigma_g$ is the monodromy.

The rays through the $\Sigma_g$ limit to the ray through $S_2$.  Thus, by the previous paragraph,
\[ (2g-2)\log(\lambda(\phi_g)) \to 2 \log(\lambda(\phi))  < \infty. \]
It follows that $(2g-2)\log(\lambda(\phi_g))$ is bounded from above by some constant $L_0$, independent of $g$, and thus $\log(\lambda(\phi_g)) < L_0/g$ for all sufficiently large $g$.  By increasing $L_0$ if necessary, we can accommodate the finitely many genera not covered by this construction, and Penner's theorem follows.

\bigskip

\noindent
{\bf Remark.}
By investigating the monodromies of specific finite-volume fibered hyperbolic $3$-manifolds, Hironaka \cite{Hir}, Aaber--Dunfield \cite{AD}, and Kin--Takasawa \cite{KT} showed that $L_0 = \log((3+\sqrt{5})/2)$ suffices for all sufficiently large $g$, as mentioned in the introduction.

\section{Counting short geodesics II}

In this section we prove Theorem~\ref{T:count2}, which states that, given $d > 0$, there exists a polynomial $Q_d(g)$ of degree $d$, with positive leading coefficient, and $L>0$ so that
\[ \left| \G_g(L) \right| \geq Q_d(g) \]
for all $g \geq 1$.

First, we require a lemma.

\begin{lemma} \label{L:big bundle small norm}
For any $g \geq 2$, there exists a pseudo-Anosov homeomorphism $\phi:S_g \to S_g$ that acts trivially on $H_1(S_g;\R)$ and has the following property: if $\R_+ \cdot F \subset H^1(M;\R)$ is the cone on the fibered face containing $S_g$, and $\psi:H^1(M_\phi;\R) \to \R$ is the homomorphism that restricts to the Thurston norm on $\R_+ \cdot F$, then $\psi(H^1(M;\Z)) = 2 \Z$.
\end{lemma}

\begin{proof}

Assume first that $g \geq 6$, and let
\[ \alpha_0,\ldots,\alpha_{m},\beta_0,\ldots,\beta_{n},\gamma \]
be the simple closed curves in $S_g$ shown here for the case $g=8$:

\vspace{.25in}

\labellist
\small\hair 2pt
\pinlabel {$\gamma$} [ ] at 120 58
\pinlabel {$\beta_0$} [ ] at 95 100
\pinlabel {$\alpha_0$} [ ] at 95 -7
\pinlabel {$\alpha_1$} [ ] at 22 70
\pinlabel {$\beta_1$} [ ] at 165 15
\pinlabel {$\alpha_2$} [ ] at 200 17
\pinlabel {$\beta_2$} [ ] at 255 16
\pinlabel {$\alpha_3$} [ ] at 334 22
\pinlabel {$\beta_3$} [ ] at 349 75
\pinlabel {$\alpha_4$} [ ] at 445 22
\pinlabel {$\beta_4$} [ ] at 291 80
\endlabellist
\noindent \includegraphics[scale=.77]{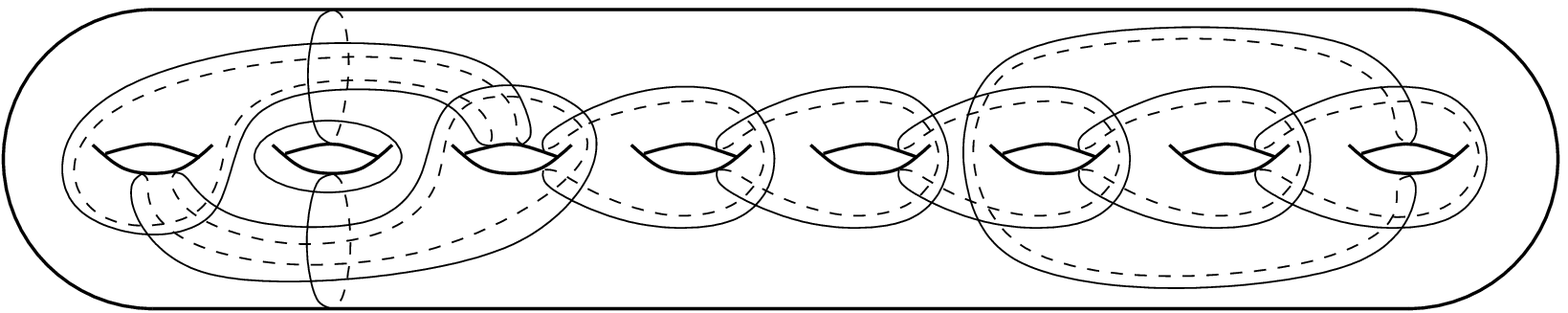}

\vspace{.25in}

\noindent (when $g$ is odd, $m=(g+1)/2$ and $n=(g-1)/2$, and when $g$ is even, $m=n=g/2$).  

Consider the product of Dehn twists:
\[ \phi = (T_{\beta_0}^{-1}T_{\alpha_0})(T_{\alpha_1} T_{\alpha_2}\cdots T_{\alpha_m})(T_{\beta_1}^{-1}T_{\beta_2}^{-1} \cdots T_{\beta_n}^{-1}). \]
By Thurston's theorem \cite[Theorem 7]{ThMCG}, the conjugate $T_{\beta_0}\phi T_{\beta_0}^{-1}$ is isotopic to a pseudo-Anosov homeomorphism (see also \cite{Pe2}), hence $\phi$ is.

The action of $\phi$ on $H_1(S_g;\R)$ is trivial.  Indeed, since each $\alpha_i$ and $\beta_i$ is separating in $S_g$ for $i > 0$, each of these $T_{\alpha_i}$ and $T_{\beta_i}$ acts trivially on $H_1(S_g;\R)$.  Also, since $\alpha_0$ and $\beta_0$ can be oriented so that they represent the same element of $H_1(S_g;\Z)$, the twists $T_{\alpha_0}$ and $T_{\beta_0}$ have the same action on $H_1(S_g;\R)$.

We will require one further property of $\phi$, which is that $\gamma$ and $\phi(\gamma)$ cobound an embedded genus $1$ surface $\Sigma_0$ in $S_g$.  To check this, note that $\phi(\gamma) = T_{\beta_0}^{-1}T_{\alpha_0}(\gamma)$.

We can construct similar configurations of curves, and hence a similar $\phi$, when $g$ is  3, 4, or 5.  Indeed, for $g=3$, we can simply use the curves $\alpha_0$, $\beta_0$, $\alpha_1$, and $\beta_1$.  The cases of $g=4$ and $g=5$ require nontrivial modifications.  However, since these cases are not logically needed for the proof of Theorem~\ref{T:count2}, we leave the constructions to the reader.  For the last case, $g=2$, any pseudo-Anosov $\phi$ acting trivially on $H_1(S_2;\Z)$ suffices.

The surface $\Sigma_0 \subset S_g \subset M_\phi = M$ is transverse to the suspension flow $\phi_t$ on $M$ since $S_g$ is.  We can push $\Sigma_0$ along flow lines to construct a closed embedded surface of genus $2$ transverse to $\phi_t$ as follows (compare [CLR], for example).  Let $N(\gamma) \subset \Sigma_0$ be a collar neighborhood of the boundary component $\gamma$.  Let $\eta:\Sigma_0 \to [0,1]$ be a smooth function supported on $N(\gamma)$, where $\eta^{-1}(1) = \gamma$, and where the derivative of $\eta$ vanishes on $\gamma$.  Define $f:\Sigma_0 \to M_\phi$ by $f(x)= \phi_{\eta(x)}(x)$.  The map $f$ is an embedding on the interior of $\Sigma_0$ and has $f(\gamma) = \phi(\gamma)$.  The image is a closed genus 2 surface $\Sigma$ which is transverse to $\phi_t$.

Let $\tau$ denote the $2$-plane bundle on $M$ defined by the tangent space to the fibers of the fibration $M \to S^1$, and let $e(\tau)$ denote its Euler class.  The restriction of $\tau$ to $\Sigma$ is homotopic to the tangent plane bundle, and hence $\psi([\Sigma]) = - e(\tau)\cdot [\Sigma]  = -\chi(\Sigma) = 2$.  Since $\psi$ is a homomorphism, its image contains $2\Z$ as desired.  
\end{proof}

\begin{proof}[Proof of Theorem~\ref{T:count2}]

Without loss of generality, suppose $d \geq 4$ is even.  Let $S$ be a closed surface of genus $d/2$, let $\phi : S \to S$ be a pseudo-Anosov homeomorphism as in Lemma~\ref{L:big bundle small norm}, and set $M = M_\phi$.   Since $\phi$ acts trivially on $H_1(S;\R)$, we have $b_1(M) = d + 1$.   Let $F$ be the fibered face of $M$ with $S \in \R_+ \cdot F$ and let $\psi:H^1(M;\R) \to \R$ be the homomorphism agreeing with the Thurston norm on $\R_+ \cdot F$.  By Lemma~\ref{L:big bundle small norm}, we have $\psi(H^1(M;\Z)) = 2\Z$.

Choosing a basis for $H^1(M;\Z)$ induces an isomorphism $H^1(M;\Z) \cong \Z^{d+1}$ which extends to an isomorphism $H^1(M;\R) \cong \R^{d+1}$.  We choose a basis for $H^1(M;\Z)$ so that, with respect to this isomorphism, $\psi$ is given by
\[ \psi(x_0,\ldots,x_d) = 2 x_0 \]
It follows that the face $F$ is contained in the hyperplane $x_0 = 1/2$.

Let $K$ be a closed $d$-cube in $F$.  If $K$ is centered at $(1/2,t_1,\ldots,t_d)$ and has side length $2r$, then
\[ K = \{ (1/2,x_1,\ldots,x_d) \in \R^{d+1} \mid \max_{j=1,\ldots,d} {|x_j - t_j| \leq r} \}.\]
Since $K$ is compact, the function $\Lambda_F$ from Theorem~\ref{T:frieddilatation} attains a maximum $C$ on $K$.  Since $\| \cdot \|\log(\Lambda_F(\cdot))$ is constant on rays (Theorems~\ref{T:thurston} and~\ref{T:frieddilatation}), the function $\| \cdot \| \log(\Lambda_F(\cdot))$ restricted to $\R_+ \cdot K$ also has some maximum $L$.  Thus, the monodromy of every primitive integral point in $\R_+ \cdot K$ is an element of $\Psi(L)$, and so corresponds to an element of $\cup_g \G_g(L)$.  

Now, given $g \geq 2$ the set
\[ \Omega_g = \{ v \in \Z^{d+1} \cap (2g-2) \cdot K \mid v \mbox{ primitive} \} \]
determines a set of fibers of $M$ with monodromies defining geodesics in $\G_g(L)$.   
Two different fibers may define the same geodesic in $\G_g(L)$, but only if the monodromies are conjugate.  In this case there is a self-homeomorphism of $M$ that sends one fiber to the other.  Such a homeomorphism induces an nontrivial isometry of the Thurston norm on $H^1(M;\R)$.  Since the unit ball is a polyhedron, this symmetry group is finite of some order $N$ (compare \cite[Corollary of Theorem 1]{ThNorm}), and so the map from $\Omega_g$ to $\G_g(L)$ is at most $N$ to $1$.

It remains to show that $|\Omega_g|$ is bounded from below by a degree $d$ polynomial with positive leading coefficient.  This is a standard counting argument (cf. \cite[Theorem 3.9]{Ap}), and so we content ourselves to explain the idea.  Before we begin, we notice that it is enough to check this for large $g$; the statement for arbitrary $g$ is then obtained by subtracting a constant from the polynomial.

We can explicitly describe the $d$-cube $(2g-2) \cdot K$ as:
\[ (2g-2) \cdot K = \left \{ (g-1,n_1,\ldots,n_d) \mid \max_{j=1,\ldots,d} |n_j - (2g-2)t_j| \leq (2g-2)r \right \}.\]
For $g$ large enough, the number of integral points in this cube is approximately $((4g-4)r)^d$.

If an integral vector $(g-1,n_1,\ldots,n_d) \in (2g-2) \cdot K$ is imprimitive, it must be divisible by one of the prime factors $p$ of $g-1$, and hence must be the $p$th multiple of an integral point in the $d$-cube $(2g-2)/p \cdot K$, which, for large $g$, contains approximately 
\[ \left (\frac{(4g-4)r}{p}\right)^d. \]
integral points.  Now, if $p_1,\ldots,p_m$ are the prime divisors of $g-1$, it follows that
\begin{align*}
|\Omega_g| &\sim ((4g-4)r)^d - \sum_{i=i}^m  \left( \frac{(4g-4)r}{p_i} \right)^d \\
&= ((4g-4)r)^d \left( 1 - \sum_{i=i}^m \frac{1}{p_i^d} \right) \\
&\geq ((4g-4)r)^d \left( 1 - \sum_{n=i}^\infty \frac{1}{n^d} \right) \\
&= C ((4g-4)r)^d.
\end{align*}
Since $d \geq 4$, we have $C > 0$, and we are done.
\end{proof}

\section{Comparing quadratic differentials}
\label{S:continuity}

Let $M$ be a finite-volume hyperbolic 3-manifold, and let $F$ be a fibered face.  To prove Theorem~\ref{T:location} we will need to see how the $3$-manifold $M$ influences the geometry of the surfaces lying over the axis for the monodromy $\phi$ of a fiber $S \in \R_+ \cdot F$.  We will need uniform control on the geometry of these surfaces as we vary the fibers.

Each monodromy of $M$ acts on Teichm\"uller space by translation along an axis.  Each such axis is defined by a quadratic differential on some Riemann surface.  The goal of this section is to describe a construction of McMullen \cite{Mc} that provides a bridge between the $3$-manifold $M$ and the quadratic differentials corresponding to its various fibers.

Let $\Gamma = \pi_1(M)$ and $\Gamma_0 \triangleleft \Gamma$ be the kernel of the abelianization, modulo torsion:
\[ 1 \to \Gamma_0 \to \pi_1(M) \to H_1(M;\Z)/\text{torsion} \to 1 .\]
Let $\widetilde M \to M$ denote the cover of $M$ associated to $\Gamma_0$.  Let $S \in \R_+ \cdot F$ be a connected fiber with monodromy $\phi:S \to S$.  The fibration $S \to M \to S^1$ lifts to a fibration over the universal covering $\mathbb R \to S^1$:
\[ \xymatrix{ \widetilde S \ar[r] \ar[d] & \widetilde M \ar[r]\ar[d] & \mathbb R \ar[d] \\
S \ar[r] & M \ar[r] & S^1}\]
The fiber $\widetilde S$ is a connected cover of $S$---in fact, it is precisely the cover corresponding to the $\phi$-invariant subspace of $H^1(S;\Z)$.

Let $\widetilde \phi_t$ denote the lift to $\widetilde M$ of the suspension flow on $M$ associated to $\phi$.  There is a product structure
\[  \widetilde M \cong \widetilde S \times \mathbb R;\]
indeed, the map $(x,t) \mapsto \widetilde \phi_t(x)$ gives a homeomorphism $\widetilde S \times \mathbb R \to \widetilde M$.

Pulling back the foliations $\F_\pm$  produces foliations $\widetilde \F_\pm$ on $\widetilde S$, and we can suspend these by $\widetilde \phi_t$ to produce foliations $\widetilde \F_\pm^M$ on $\widetilde M$.  Alternatively, $\widetilde \F_\pm^M$ is obtained by pulling back $\F_\pm^M$ to $\widetilde M$.

Let $\pi:\widetilde M \to \widetilde S$ denote the map obtained by collapsing each flow line of $\widetilde \phi_t$ to a point:
\[ \pi(\widetilde \phi_t(x)) = x.\]

Let $\Sigma$ be a fiber in $\R_+ \cdot F$.  By Theorem~\ref{T:friedtransverse}, we can assume that $\Sigma$ is transverse to $\phi_t$. Next, let $\widetilde \Sigma$ be one component of the preimage of $\Sigma$ in $\widetilde M$.  
The first return map of $\phi_t$ is the monodromy $\varphi:\Sigma \to \Sigma$, and from this one can show that $\pi|_{\widetilde \Sigma}:\widetilde \Sigma \to \widetilde \Sigma$ is a homeomorphism; see \cite[Corollary 3.4]{CLR}.  Since the stable and unstable foliations for $\Sigma'$ are obtained by intersecting $\Sigma$ with $\F_\pm^M$, it follows that this homeomorphism sends the leaves of the lifts of the stable and unstable foliations on $\widetilde \Sigma$ to those on $\widetilde S$.  That is, for every connected fiber $\Sigma \in \R_+ \cdot F$, we have identified the cover $\widetilde \Sigma$ homeomorphically with a fixed connected covering $\widetilde S$ of $S$ so that the preimages of the stable and unstable foliations for $\varphi$ under this identification pull back to $\widetilde \F_\pm$.

The monodromy $\varphi$ for $\Sigma$ does determine a pair of transverse measures $\mu_\pm(\Sigma)$ (unique up to scaling) on the stable and unstable foliations, respectively.  This defines a complex structure on $\Sigma$ and a holomorphic quadratic differential $q(\Sigma)$ for which the vertical and horizontal measured foliations are $\mu_\pm(\Sigma)$, respectively.  Furthermore, $q(\Sigma)$ defines the axis for $\varphi$ on Teichm\"uller space $\T(\Sigma)$.  Pulling $q(\Sigma)$ back to $\widetilde S$, we have a complex structure and holomorphic quadratic differential we denote $\widetilde q(\Sigma)$ on $\widetilde S$ whose vertical and horizontal foliations are precisely $(\widetilde \F_\pm,\widetilde \mu_\pm(\Sigma))$, where $\widetilde \mu_\pm(\Sigma)$ are the measures $\mu_\pm(\Sigma)$ pulled back to $\widetilde S$.

McMullen extends this construction of a complex structure and quadratic differential in a continuous way to every point of $\R_+ \cdot F$, not just the fibers \cite{Mc}.  More precisely, let $Q(\widetilde S,\widetilde{\F}_\pm)$ denote the set of pairs consisting of a complex structure on $\widetilde S$ together with a holomorphic quadratic differential for which the horizontal and vertical foliations are $\widetilde \F _\pm$.  We denote a point of $Q(\widetilde S,\widetilde \F_\pm)$ by $\widetilde q$, suppressing the complex structure in the notation.  An element $\widetilde q \in Q(\widetilde S,\widetilde \F_\pm)$ determines a Euclidean cone metric for which the leaves of $\widetilde{\F}_\pm$ are geodesics (with the leaves of $\widetilde {\F}_+$ orthogonal to those of $\widetilde {\F}_-$), and by an abuse of notation we denote this metric $\widetilde q$.  We topologize $Q(\widetilde S,\widetilde{\F}_\pm)$ with the topology of locally uniform convergence of these metrics.  Specifically, a sequence  $\{q_n\} \subset Q(\widetilde S)$ converges to $q \in Q(\widetilde S)$ if for any compact set $K \subset \widetilde S$, $q_n:K \times K \to \mathbb R$ converges uniformly to $q:K \times K \to \mathbb R$.  

The main consequence of McMullen's work that we will need is the following.

\begin{theorem} \label{T:mcmullen}
There is a continuous map $\widetilde \q:\R_+ \cdot F \to Q(\widetilde S,\widetilde{\F}_\pm)$ which is constant on rays, and has the property that for every fiber $S \in \R_+ \cdot F$, $\widetilde \q(S) = \widetilde q(S)$, up to scaling and Teichm\"uller deformation.
\end{theorem}

The map $\widetilde \q$ is given in \cite[Theorem 9.3]{Mc}, though it is only defined up to scaling and Teichm\"uller deformation, and so one must make some choices to obtain a well-defined map.  This can be done, for example, by choosing a rectangle with sides in $\widetilde {\F}_\pm$, and then for any $\eta \in \R_+ \cdot F$, we normalize the quadratic differential $\widetilde \q(\eta)$ by requiring that the side lengths are both $1$.  Continuity follows from the description in terms of train tracks:  the horizontal and vertical foliations are carried by train tracks $\widetilde \tau_\pm$ on $\widetilde S$, and the weights on the branches determined by the vertical and horizontal foliations for $\widetilde \q$ are given by eigenvectors of a continuously varying family of Perron-Frobenius matrices and appropriate equivariance conditions; see the proof of \cite[Theorem 8.1]{Mc}.  (Our normalization convention can be chosen to correspond to a normalization in the Perron-Frobenius eigenvectors).  Since the charts for the Euclidean metric are obtained by integrating these two measures, and since the measures vary continuously, so do the metrics.

\section{Dehn filling}
\label{S:df}

Let $M$ be a finite-volume cusped hyperbolic 3-manifold with $r$ cusps.  Let $\widehat M$ denote the manifold obtained by removing the interiors of the cusps, so that $\widehat M$ is a compact manifold with $r$ boundary components $\partial_1 \widehat M,\ldots,\partial_r \widehat M$, each homeomorphic to a torus. 

A \emph{slope} on $\partial_i \widehat M$ is either the isotopy class of an unoriented essential simple closed curve in $\partial_i \widehat M$ or $\infty$.  If we choose a basis for $\pi_1(\partial_i \widehat M) \cong \Z^2$, then a slope $\beta_i \neq \infty$ corresponds to a coprime pair of integers $\beta_i = (p_i,q_i)$, unique up to sign.

Suppose $\beta = (\beta_1,\ldots,\beta_r)$ is a choice of slopes in $\partial_1 \widehat M,\ldots,\partial_r \widehat M$, respectively.  The $\beta$-\emph{Dehn filling} of $M$ is the 3-manifold $M(\beta)$ obtained from $\widehat M$ by the following procedure:
\begin{itemize}
\item For each $i$ with $\beta_i \neq \infty$, we glue a solid torus $S^1 \times D^2$ to $\partial_i \widehat M$ so that the curve $\{*\} \times \partial D^2$ represents $\beta_i$.
\item For each $i$ with $\beta_i = \infty$, we reglue the original cusp (or, what is the same thing, we can leave that cusp alone from the start).
\end{itemize}
The homeomorphism type of $M(\beta)$ depends only on $\beta$.

We can view the set of slopes on $\partial_i \widehat M$ as points in $\Z^2 \cup \{\infty\} \subset \R^2 \cup \{\infty\} \cong S^2$.  We say that a sequence of slopes $\{\beta_i^n\}_{n=1}^\infty$ on $\partial_i \widehat M$ tends to $\infty$ if it does in $\R^2 \cup \{\infty\}$.

The inclusion $\widehat M \to M$ induces an isomorphism $\pi_1(M) \cong \pi_1(\widehat M)$.  If we compose this isomorphism with the homomorphism $\pi_1(\widehat M) \to \pi_1(M(\beta))$ induced by inclusion, we obtain a canonical homomorphism $\pi_1(M) \to \pi_1(M(\beta))$.  By Van Kampen's theorem, this is surjective.

The next result, Thurston's Dehn surgery theorem \cite[Theorem 5.8.2]{Th}, states that Dehn filling on a hyperbolic $3$-manifold usually produces a hyperbolic $3$-manifold.
\begin{theorem} \label{T:Dehn surgery theorem}
Let $M$ be a finite-volume hyperbolic 3-manifold with $r$ cusps.  Suppose $\{ \beta^n\}_{n=1}^\infty$ is a sequence of $k$-tuples of slopes with $\beta^n = (\beta^n_1,\ldots,\beta^n_r)$ and $\beta_i^n$ a slope on $\partial_i \widehat M$ for all $i$ and $n$.  Assume that for all $i$, $\beta_i^n$ tends to $\infty$.  Then $M(\beta^n)$ is hyperbolic for all but finitely many $n$.  

Moreover, for appropriate choices of the holonomy homomorphisms $\pi_1(M),\pi_1(M(\beta^n)) \to \PSL_2(\C)$ within the respective conjugacy classes, the composition
\[ \pi_1(M) \to \pi_1(M(\beta^n)) \to \PSL_2(\C) \]
converges pointwise to
\[ \pi_1(M) \to \PSL_2(\C).\]
\end{theorem}

We will require the following simple application of Theorem~\ref{T:Dehn surgery theorem}.

\begin{corollary} \label{C:Dehn Surgery Corollary}
Let $M$ be a finite-volume hyperbolic 3-manifold with $r$ cusps and let $\alpha \in \pi_1(M)$ be any nontrivial element.  For each $i = 1,\ldots,r$, there are finitely many slopes $\beta_i^1,\ldots,\beta_i^{s_i}$ on $\partial_i \widehat M$, $\beta_i^{j} \neq \infty$ for all $j$, so that if $\beta_1,\ldots,\beta_r$ are slopes with $\beta_i \neq \beta_i^j$ for each $i = 1,\ldots,r$ and $j = 1,\ldots,s_i$, then $\alpha$ represents a nontrivial element of $\pi_1(M(\beta_1,\ldots,\beta_r))$. 
\end{corollary}

\section{Lengths of curves}

In this section we recall three notions of length for a simple closed curve in a surface $S$ equipped with a complex structure $X$, and we recall various well-known relationships between them.  Here, and throughout, we say that a simple closed curve is \emph{essential} if it is homotopic neither to a point nor a puncture.  For further details, see \cite{Ah1,Ah2,Ma,Wo}.

\p{Extremal length.} A Borel metric on $S$ with respect to $X$ is a metric that is locally given by $\rho(z)|dz|$, where $\rho \geq 0$ is a Borel measurable function and $z$ is a local coordinate for the complex structure $X$.

Let $\alpha$ be a simple closed curve in $S$.  The {\em extremal length} of $\alpha$ with respect to $X$ is
\[ \ext_X(\alpha) = \sup_{\rho} \frac{L_\rho(\alpha)^2}{\Area(\rho)}, \]
where the supremum is over all Borel metrics $\rho$ in the conformal class of $X$, $L_\rho(\alpha)$ is the infimum of $\rho$-lengths of closed curves in the homotopy class of $\alpha$, and $\Area(\rho)$ is the area of $S$ with respect to $\rho$.

\p{Modulus.} Another number associated to $\alpha$ with respect to $X$ is the {\em modulus}.  Recall that if an annulus $A$ is conformally equivalent to $\{ z \in \C : 1 < |z| < R \}$, then the modulus of $A$ is $m_X(A) = \log(R)/2\pi$.   The modulus of a simple closed curve $\alpha$ is defined as
\[  m_X(\alpha) = \sup_{A \supset \alpha} m_X(A), \]
where the supremum is taken over all embedded annuli in $S$ containing a curve homotopic to $\alpha$.
When $\alpha$ is inessential, then $m_X(\alpha)=\infty$. 

We can alternatively define the modulus of $A$ via extremal lengths:
\[ m_X(A) = \sup_{\rho} \frac{L_\rho(A)^2}{\Area(\rho)}, \]
where the supremum is over all Borel metrics $\rho$ in the conformal class of $X$, and where $L_\rho(A)$ is the infimum of the $\rho$-lengths of all paths in $A$ connecting distinct boundary components.

\p{Modulus versus extremal length.} The relationship between modulus and extremal length is provided by the following; see, for example, \cite[Section 1.D]{Ah2}.

\begin{proposition} \label{P:extremal and modulus}
Let $\alpha$ be a simple closed curve in $S$ and $X$ a complex structure on $S$.  We have
\[ \ext_X(\alpha) = 1/m_X(\alpha).\]
\end{proposition}

\p{Hyperbolic length.} There is a third measurement associated to a closed curve $\alpha$ with respect to $X$.  Suppose $(S,X)$ can by uniformized as a quotient of the hyperbolic plane (for example, if $\chi(S) < 0$), that is, there is a conformal homeomorphism between $(S,X)$ and a quotient of the hyperbolic plane by a discrete, torsion-free subgroup of the orientation-preserving isometry group.  In a hyperbolic surface, every essential closed curve has a unique geodesic representative.  The length of the geodesic representative of $\alpha$ is thus an invariant of $\alpha$ that we denote $\ell_X(\alpha)$.  This is called the \emph{hyperbolic length} of $\alpha$.

\p{Hyperbolic collars}  Keen's collar lemma \cite{Ke} provides a quantitative lower bound on the width of an annular neighborhood of a simple closed geodesic in a hyperbolic surface.  From this one obtains lower bounds on the length of a curve intersecting the given curve.  This is stated conveniently in terms of the geometric intersection number $i(\alpha,\beta)$ for a pair of simple closed curves $\alpha$ and $\beta$.

\begin{lemma} \label{L:collar}
There is a function $F:\R_+ \to \R_+$ that satisfies
\[ \displaystyle{\lim_{x \to 0} F(x) = \infty} \]
and also satisfies the following property: if $\alpha$ and $\beta$ are simple closed curves in $S$, then for any $X \in \T(S)$ we have
\[ \ell_X(\beta) \geq i(\alpha,\beta)F(\ell_X(\alpha)). \]
\end{lemma}

In fact, we can take the function $F$ from Lemma~\ref{L:collar} to be
\[ F(x) =  2\sinh^{-1}\left( \frac{1}{\sinh(x/2)} \right); \]
see, for example, \cite[Lemma 13.6]{primer}.

\p{Hyperbolic length versus extremal length.}  Since a hyperbolic metric on $S$ is Borel, the hyperbolic length of a curve can be related to its extremal length directly from the definition.  The following result of Maskit gives stronger bounds, independent of the topology of $S$ \cite{Ma}.

\begin{proposition} \label{P:hyperbolic and extremal}
For $\alpha$ an essential closed curve in $S$ we have
\[ \frac{\ell_X(\alpha)}{\pi} \leq \ext_X(\alpha) \leq \frac{\ell_X(\alpha)}{2} e^{\ell_X(\alpha)/2}.\]
\end{proposition}

\p{Hyperbolic lengths with respect to different complex structures} The following result of Wolpert \cite[Lemma 3.1]{Wo} relates the distance in Teichm\"uller space to distortion of hyperbolic lengths.

\begin{proposition} \label{P:wolpert}
Given $X,Y \in \T(S)$ and $\alpha$ an essential closed curve in $S$, we have
\[\ell_X(\alpha) \leq e^{d_\T(X,Y)} \ell_Y(\alpha).\]
\end{proposition}

The next fact, sometimes called the Schwarz--Pick--Ahlfors lemma, states that a holomorphic mapping is a contraction with respect to the hyperbolic metrics on domain and range \cite[Theorem A]{Ah1}.
\begin{theorem} \label{T:Schwarz-Pick}
If $f:S \to S'$ is a holomorphic mapping with respect to complex structures $X$ and $Y$ on surfaces $S$ and $S'$, respectively, then $f$ is a contraction with respect to the hyperbolic metrics on the domain and range.  In particular,
\[ \ell_Y(f(\alpha)) \leq \ell_X(\alpha)\]
for any closed curve $\alpha$ in $S$.
\end{theorem}

\section{Location of short geodesics} \label{S:locationProof}

We are now ready to prove Theorem~\ref{T:location}, which states that, given $L > 0$, there exists $\epsilon_2 > \epsilon_1 > 0$ so that, for each $g \geq 1$, we have
\[ \G_g(L) \subset \M_{g,[\epsilon_1,\epsilon_2]}.\]

Propositions~\ref{P:thick} and~\ref{P:thin} below give the containments $\G_g(L) \subset \M_{g,[\epsilon_1,\infty)}$ and $\G_g(L) \subset \M_{g,(0,\epsilon_2]}$, respectively.  

\begin{proposition} \label{P:thick}
Let $L > 0$.  There exists $\epsilon > 0$ so that
\[ \G_g(L) \subset \M_{g,[\epsilon,\infty)}.\]
\end{proposition}

\begin{proof}

First of all, since $\G_1(L)$ is finite, it suffices to prove the proposition for $g \geq 2$.  Indeed, we can take $\epsilon$ to be the minimum of the $\epsilon$'s obtained for $g=1$ and for $g \geq 2$, respectively.

Let $[\phi] \in \Psi_L(S)$ and let $X \in \T(S)$ be a point on the axis for $\phi$.  Let $\gamma$ denote the essential closed curve with shortest length $\ell_X(\gamma)$.  We must find a uniform lower bound $\epsilon$ for $\ell_X(\gamma)$.

Let $F(x)$ be the function from Lemma~\ref{L:collar}, and let $\epsilon > 0$ be such that
\[ F(x) > e^{3L}x \]
for every $x < \epsilon$.  We will show that $\ell_X(\gamma) \geq \epsilon$.

Say that the genus of $S$ is $g \geq 2$.  Any collection of pairwise disjoint, homotopically distinct, essential simple closed curves in $S$ has cardinality at most $3g-3$.  Thus,  for some $k \leq 3g-2$ we have $i(\phi^k(\gamma),\gamma) \neq 0$.  By Lemma~\ref{L:collar},
\[ \ell_X(\phi^k(\gamma)) \geq F(\ell_X(\gamma)).\]
On the other hand, by Proposition~\ref{P:wolpert}, we have
\[ \ell_X(\phi^k(\gamma)) \leq \lambda(\phi^k) \ell_X(\gamma) = \lambda(\phi)^k \ell_X(\gamma) \leq \lambda(\phi)^{3g-2} \ell_X(\gamma) \]
Combining the last two displayed inequalities with the fact that $3g-2 < 3g$ and the assumption that $[\phi] \in \Psi_L(g)$, we have
\begin{align*}
F(\ell_X(\gamma)) \leq \ell_X(\phi^k(\gamma)) \leq \lambda(\phi)^{3g-2} \ell_X(\gamma) < \\ \lambda(\phi)^{3g}\ell_X(\gamma)  \leq (e^{L/g})^{3g} \ell_X(\gamma) \leq e^{3L} \ell_X(\gamma).
\end{align*}
By the definition of $\epsilon$, this implies that $\ell_X(\gamma) \geq \epsilon$, as desired.
\end{proof}

The second half of Theorem~\ref{T:location} is more involved.  As the proof is in terms of pseudo-Anosov homeomorphisms rather than geodesics in $\M_g$, we explain a complete translation to that language.

Given a finite volume $3$-manifold $M$ and any subset $K \subset F$ of an open fibered face $F$ of $M$, let $\Psi(L,K)$ denote the set of closed-surface pseudo-Anosov homeomorphisms $(\phi:S \to S) \in \Psi(L)$ such that, after removing some $\phi$-invariant subset of the singular points of the stable foliation, the resulting surface $S'$ is a fiber in $\R_+ \cdot K$ with monodromy $\phi' = \phi|_{S'}$.
We emphasize that $S'$ is obtained by removing none, some, or all of the singular points of the closed surface $S$.

Given $X \in \T(S)$, let $\inj(X)$ denote the $X$-hyperbolic length of the shortest essential closed curve.  Given a pseudo-Anosov homeomorphism $\phi:S \to S$, write $\inj(\phi)$ to denote the maximum of $\inj(X)$ as $X$ varies over all complex structures in $\T(S)$ lying on the axis for $\phi$.  For any pseudo-Anosov $\phi : S_g \to S_g$ with associated geodesic $\gamma_\phi \subset \M_g$ we have:
\[ \inj(\phi) \leq \epsilon \Longleftrightarrow \gamma_\phi \subset \M_{g,(0,\epsilon]}. \]
Finally, we define
\[ \Psi_{(0,\epsilon]} = \{ \phi: S \to S : S \text{ any surface}, \inj(\phi) \leq \epsilon \}.  \]
This notation should remind the reader of $\M_{g,(0,\epsilon]}$, as each $\phi : S_g \to S_g$ in $\Psi_{(0,\epsilon]}$ corresponds to a geodesic contained in $\M_{g,(0,\epsilon]}$.

\bigskip

\begin{proposition} \label{P:thin}
For every $L > 0$, there exists $\epsilon > 0$ so that, for every $g \geq 1$, we have
\[ \G_g(L) \subset \M_{g,(0,\epsilon]}. \]
Equivalently, 
\[ \Psi(L) \subset \Psi_{(0,\epsilon]}.\]
\end{proposition}

\begin{proof}

Again, since $\G_1(L)$ is finite, it is enough to prove the proposition for $g \geq 2$.  Fix $L  > 0$.  We will prove the following statement by induction on $r$.

\medskip

\begin{quote}
\emph{Let $M$ be a hyperbolic 3-manifold with $r \geq 0$ cusps.  There is an $\epsilon(M)$ so that, for each open fibered face $F$ of $M$, we have $\Psi(L,F) \subseteq \Psi_{(0,\epsilon(M)]}$.
}
\end{quote}

\medskip

The proposition then follows by taking $\epsilon$ to be the maximum of $\epsilon(M)$, where $M$ ranges over the finite set of manifolds $\mathcal{T}(L)$ given by Theorem~\ref{T:universal}.

\bigskip

We first treat the case $r = 0$, that is, the case where $M$ is closed.  Besides serving as the base case for the induction, this case will also explain the main ideas for the more complicated inductive step.

Fix a closed, fibered, hyperbolic $M$.  Since $M$ has finitely many fibered faces (Theorem~\ref{T:thurston}), it suffices to show that, given some such face $F$, there is an $\epsilon(F)$ so that
\[ \Psi(L,F) \subseteq \Psi_{(0,\epsilon(F)]}.\]

It follows from Theorem~\ref{T:frieddilatation} that there is a compact subset $K$ of the open face $F$ with the property that:
\[ \Psi(L,F) = \Psi(L,K). \]
Thus, it suffices to show that, for any such $K$, there is an $\epsilon(K)$ so that
\[ \Psi(L,K) \subseteq \Psi_{(0,\epsilon(K)]}.\]

Fix a fibered face $F$ of $M$, and let $\Gamma_0$, $\Gamma$, $\widetilde M$, $\widetilde S$, $\widetilde{\F}_\pm$, and $\widetilde \q:\R_+ \cdot F \to Q(\widetilde S,\widetilde{\F}_\pm)$ be the objects associated to $F$ as in Section~\ref{S:continuity}.  As above, fix $K \subseteq F$ so that $\Psi(L,F)=\Psi(L,K)$.

Let $\alpha$ be any essential simple closed curve in $\widetilde S$.  We take an annular neighborhood $A$ of $\alpha$ in $S$.  By Theorem~\ref{T:mcmullen}, $\widetilde \q(K)$ is a compact subset of $Q(\widetilde S,\widetilde{\F}_\pm)$.  Thus, we obtain a uniform lower bound on the $\widetilde \q(\eta)$-distance between the boundary components of $A$ and a uniform upper bound on the $\widetilde \q(\eta)$-area, for every $\eta \in \R_+ \cdot K$.  Consequently, the modulus $m_{\widetilde\q(\eta)}(A)$ is uniformly bounded below, and hence so is $m_{\widetilde \q(\eta)}(\alpha)$.  By Proposition~\ref{P:extremal and modulus}, the $\widetilde \q(\eta)$-extremal length of $\alpha$ is then uniformly bounded from above.  Then, by Proposition~\ref{P:hyperbolic and extremal}, the $\widetilde \q(\eta)$-hyperbolic length is bounded from above.  Denote this uniform bound by $C$. 

Suppose $(\phi :S \to S) \in \Psi(L,K)$, where $S=S_g$.  (Since $M$ is closed, we do not remove any points of $S$ in order to obtain a fiber in $\R_+ \cdot K$.)  Let $p: \widetilde S \to S$ be the covering map.  The quadratic differential $\widetilde \q(S)$ descends to the quadratic differential $q(S)$ on $S$ defining the axis for $\phi$ in $\T(S)$.  Let $X_0$ be the point on the axis corresponding to the underlying complex structure of $q(S)$.   It follows that $p(\alpha)$ is a closed essential curve in $S$ with $X_0$-hyperbolic length at most $C$.

Since $\phi \in \Psi_g(L)$, we have
\[ \log(\lambda(\phi)) \leq L/g. \]
If $X \in \T(S)$ is any point along the axis of $\phi$ between $X_0$ and $\phi(X_0)$, then
\[ d(X_0,X) \leq d(X_0,\phi(X_0)) \leq L/g,\]
and hence Propostion~\ref{P:wolpert} implies
\[ \ell_X(\alpha) \leq e^{L/g} \ell_{X_0}(\alpha) \leq e^{L/g} C \leq e^L C. \]
For any other point $X \in \T(S)$ on the axis for $\phi$, there is an $n$ so that $\phi^n(X)$ lies between $X_0$ and $\phi(X_0)$, and hence
\[ \ell_X(\phi^{-n}(\alpha)) = \ell_{\phi^n(X)}(\alpha) \leq e^L C.\]
Thus, $\inj(\phi) \leq e^L C$.  As this bound is independent of the choice of $\phi : S \to S$ in $\Psi(L,K)$, we can set $\epsilon(F) = \epsilon(K)  =  e^L C$, and this completes the proof in the base case.\\

We are now ready for the inductive step.  Let $M$ be a 3-manifold with $r > 0$ cusps.  As in the base case, it suffices to focus on a single fibered face $F$ and a compact subset $K \subset F$ with $\Psi(L,F) = \Psi(L,K)$.

Let $\widetilde S$ be the common cover for all fibers in $\R_+ \cdot K$, as in Section~\ref{S:continuity}.  We can carry out the same argument as in the base case in order to find an essential curve $\alpha$ in $\widetilde S$ with $\widetilde \q(\eta)$-hyperbolic length at most $C$ for all $\eta \in \R_+ \cdot K$.  Then, for any $(\phi:S \to S) \in \Psi(L,K)$ with $S' \in \R_+ \cdot K$ the punctured fiber, if we let $X' \in \T(S')$ denote the underlying complex structure for $q(S')$, then $\ell_{X'}(p(\alpha)) \leq C$. 
Moreover, if $X \in \T(S)$ is the complex structure on $S$ obtained by filling in the punctures and extending $X'$, then by Theorem~\ref{T:Schwarz-Pick}
\[ \ell_X(p(\alpha)) \leq C.\]
Thus, as long as $p(\alpha)$ remains an essential curve after filling in the punctures, we can argue just as in the base case and prove $\inj(\phi) \leq e^L C$.  In other words, we have shown that if $(\phi:S \to S) \in \Psi(L,F)$ and $p(\alpha)$ is essential in $S$, then $\phi \in \Psi_{(0,N]}$, for $N = e^L C$.  It remains to deal with the cases where $p(\alpha)$ is inessential in $S$.  That is, we must find $N'$ so that any $(\phi:S \to S) \in \Psi(L,F)$ with $p(\alpha)$ inessential in $S$ is contained in $\Psi_{(0,N']}$.  Then we may set $\epsilon(F) = \max\{N,N'\}$.

We will first define $N'$, and then prove that it satisfies the above statement.  
For each $i = 1,\ldots,r$, let $\beta_i^1,\ldots,\beta_i^{s_i}$ be the slopes from Corollary~\ref{C:Dehn Surgery Corollary}, and define Dehn fillings
\[ M_i^j = M(\infty,\ldots,\infty,\beta_i^j,\infty,\ldots,\infty). \]
The manifold $M_i^j$ has $r-1$ cusps.  Therefore, by induction, there are real numbers $\epsilon(M_i^j)$, so that if $F$ is any fibered face of $M_i^j$, then $\Psi(L,F) \subset \Psi_{(0,\epsilon(M_i^j)]}$.  Let
\[ N' = \max \{ \epsilon(M_1^1),\ldots,\epsilon(M_r^{s_r}) \}. \]

Let $\phi:S \to S$ be an element of $\Psi(L,F)$ with $p(\alpha)$ inessential in $S$.  We must show that $\phi \in \Psi_{(0,N']}$.  The idea is to show that, up to removing singularities, $\phi$ is the monodromy for some $M_i^j$.

We can view the mapping torus $M_\phi$ as being obtained from $M$ by Dehn filling:
\[ M_\phi = M(\beta_1,\ldots,\beta_r). \]
Let $S'$ denote the fiber in the cone over $F$ corresponding to $S$; recall that $S'$ is obtained by removing from $S$ a set of singular points of the foliations for $\phi$.

Each slope $\beta_i \neq \infty$ is the intersection of $S'$ with the corresponding boundary component of the truncated manifold $\widehat M$; see Section~\ref{S:df}.  In particular it makes sense to write $\beta_i = \beta_i(S')$.

If $p(\alpha)$ is not essential in $S$, then $p(\alpha)$ must be trivial in
\[ M_\phi = M(\beta_1(S'),\ldots,\beta_r(S')) \]
and hence $\beta_i(S') = \beta_i^j$ for at least one $i \in \{1,\ldots,r\}$ and some $j  \in \{1,\ldots,s_i\}$.   It follows that the manifold $M_i^j$ defined above fibers with fiber $S''$ where $S' \subset S'' \subset S$, and $S''$ is obtained from $S'$ by adding in the $\phi$-orbit of the singular point corresponding to the $i$th cusp of $M$.

Suppose $F_i^j$ is the open face of $M_i^j$ with $S'' \in \R_+ \cdot F_i^j$.  Since $\phi \in \Psi(L,F)$, it follows that
\[ \phi \in \Psi(L,F_i^j) \subset \Psi_{(0,\epsilon(M_i^j)]} \subset \Psi_{(0,N']}, \]
as desired.
\end{proof}

\p{Remarks on the proof}
\emph{(1)} It is conceivable that one might be able to find a single curve $\alpha \in \widetilde S$ which when projected to any fiber remains essential after filling in the missing singular points, simplifying the proof, though it is not clear how to find such a curve.

\medskip

\noindent
\emph{(2)} It can happen that a punctured surface has a hyperbolically short essential closed curve, while the filled in surface has no short curves.  For example,  start with a closed surface; it has some shortest essential curve.  Next, puncture the surface at two points.  By taking these points to be close together, a curve surrounding these two punctures can have an annular neighborhood of arbitrarily large modulus, and so this curve is arbitrarily short on the punctured surface.  This short curve must become inessential when the punctures are filled back in.\\

\bibliography{geography}{}
\bibliographystyle{plain}

\end{document}